\documentclass{commat}

\title{Concurrent normals of immersed manifolds}

\author{Gaiane Panina, Dirk Siersma}

\authorinfo[G. Panina]{St. Petersburg Department of V.A.Steklov Institute of Mathematics RAS, St. Petersburg, Russia}{gaiane-panina@rambler.ru}

\authorinfo[D. Siersma]{Mathematisch Instituut, Universiteit Utrecht, The Netherlands}{d.siersma@uu.nl}

\abstract{It is conjectured since long that for any
convex body  $K \subset \mathbb{R}^n$ there exists a point in
the interior of $K$ which belongs to at least $2n$ normals from different points on the
boundary of $K$. The conjecture is known to be true for $n=2,3,4$.

Motivated by a recent results of Y. Martinez-Maure, and an approach by A. Grebennikov and G. Panina, we prove the following: Let a compact smooth $m$-dimensional manifold $M^m$ be  immersed in  $ \mathbb{R}^n$. We assume that    at least  one of the homology groups $H_k(M^m,\mathbb{Z}_2)$  with $k<m$  vanishes. Then   under mild conditions, almost every normal line
to $M^m$
contains an intersection point  of at least $\beta +4$ normals from different points of $M^m$, where $\beta$ is the  sum of Betti numbers of $M^m$.}



\VOLUME{31}
\NUMBER{3}
\YEAR{2023}
\firstpage{1}
\DOI{https://doi.org/10.46298/cm.10840}

\begin{document}

\section{Motivations}\label{motiv}

 Our first motivation comes from the case when $M^{n-1}=\partial K$ is the boundary of some convex body $K \subset \mathbb{R}^n$. It is conjectured since long that there exists a~point in
the interior of $K$ which is the intersection point of at least $2n$ normals from different points on the
boundary of $K$.
The conjecture trivially holds for $n=2$. For $n=3$ it was proven by E. Heil (see~\cite{H1} and~\cite{H2}) via geometrical methods and reproved by Pardon via topological methods.
The case $n=4$ was completed also by J. Pardon in~\cite{P}.

Recently Y. Martinez-Maure proved in~\cite{M-M} for $n=3,4$ that (under some mild conditions) almost every normal
through a~boundary point to a~smooth convex body $K$
 passes arbitrarily close to the set
of points lying on normals through at least six distinct
points of $\partial K$. This result was reproved by A. Grebennikov and G. Panina in~\cite{GP}. Extending their method, we prove related results for concurrent normals of immersed smooth compact manifolds.

Another motivation comes from \textit{tight embeddings} of manifolds.
A tight embedding of $M^m$ is such that almost every linear function has exactly $\beta(M^m, \mathbb{Z}_2)$ critical points, which is the minimal possible amount.

In this paper we consider immersions of $m$-dimensional manifolds $M^m$ into $\mathbb{R}^n$. Critical points of the squared distance function $d_y^2$ corresponds to normals through $y$.
Linear functions can be viewed as the limit to infinity of $d_y^2$.

It is also worth mentioning that there is a~close
relation between
the number of
critical points of linear functions and
the curvature of the
embedding~\cite{Kui}.

\section{The geometry of the focal set}\label{s:focal}

Consider an immersion $j: M^m \rightarrow \mathbb{R}^n$ of a~smooth $m$-dimensional connected manifold without boundary. In this paper we will omit $j$ from the notation and write $M^m \subset \mathbb{R}^n$ instead. However, we emphasize here that all constructions depend on the immersion.

\begin{definition}
Given an immersed manifold $M^m$, the \emph{normal} $\mathcal{N}_p$ at a~point $p\in M^m$ is a~line passing through $p$ and orthogonal to $M^m$ at the point $p$.    
\end{definition}

The focal set of the immersed surface $M^m$ arises naturally in this paper. We will use its definitions and properties borrowed from~\cite{M},~\cite{Po},~\cite{CR1} and~\cite{CR2} (the last two references also apply for immersions in spherical and hyperbolic spaces).
Following the standard approach, we make use of the squared distance function to $M^m$:
\[
d^2_y: M^m \to \mathbb{R}, \ \ d(x) =||x-y ||^2.
\]

Given a~point $y\in \mathbb{R}^n$, all the normals passing through $y$ can be read off the function $d_y^2$:
\textit{ A point $y$ lies on the normal $\mathcal{N}_p$ for some $p\in M^m$
  iff
   $p$ is a~critical point of the function $d_y^2$.} So, the number of concurrent normals through a~point $y$ is equal to the number of critical points of $d_y^2$.

By Sard's theorem type arguments (\cite{M}, Theorem 6.6), $d_y^2$ is a~Morse function for almost all $y$.

\begin{definition}
The bifurcation diagram of $d_y^2$ (cf.~\cite{A}) is called the \emph{focal set} $\mathcal{F}_M$ of $M^m$. In other words, the focal set coincides with the set of points $y \in \mathbb{ R}^n$ where the function $d^2_y$ is not a~Morse function (see~\cite{M}).

 An equivalent definition (appearing in~\cite{M}) is the following: the focal set $\mathcal{F}_M$ is the set of critical values of the evaluation map
\[
e: \mathcal{N(}M^m) \rightarrow \mathbb{R}^n \; \; \;\; \; \; \; e(p,t \mathbf{n}) =p + t \mathbf{n}
\]
from the normal bundle of $M^m$, where $p \in M^m$ and $\mathbf{n}$ is a~unit normal vector at $p$. 
\end{definition}

Note that $e$ is a~map between manifolds of the same dimension. We will discuss below the structure of its critical set $\mathcal{C}_{e}$.

 The elements of $\mathcal{F}_M$ are called \textit{focal points}. They generalize the centers of principal curvatures of hypersurfaces. Generically the dimension of $\mathcal{F}_M$ is $n-1$, but for instance, in the case of the round sphere it degenerates to a~point.

\begin{definition}
 For all $y \in \mathcal{N}_p$, the point $p$ is critical for $d^2_y$, and for almost all $y\in \mathcal{N}_p$ the critical point $p$ is non-degenerate.
We say that a~point $y\in \mathcal{N}_p$ is a~\textit{$p$-focal point} if $p$ is a~degenerate critical point of the function $d^2_y$.
\end{definition}

The Morse index of $d^2_y$ at $p$ for a~non-$p$-focal point $y\in \mathcal{N}_p$ is equal to the number the $p$-focal points lying between $p$ and $y$, counted with multiplicity. The multiplicity $\nu$ is defined as the nullity of the Jacobian of $e$ at the point $(x, t \mathbf{n})$, which equals to the nullity of the 2-jet of $d^2_y$ at $y = x+ t \mathbf{n}$~\cite{M}.

Let $y=p+t\mathbf{n}$ be a~point on $\mathcal{N}_p$. Once the normal vector $\mathbf{n}$ is fixed, the set of $p$-focal
points on the normal becomes linearly ordered.
We enumerate them as follows: $r_j$ is the $j$-th $p$-focal point contained in $\{t\geq0\}$ part of $\mathcal{N}_p$, counting from the point $p$, and $r_{-j}$ is the $j$-th $p$-focal point contained in $\{t\leq 0\}$ part, also counting from the point $p$, counted with multiplicity.

We shall also consider the \textit{compactified normal line} $\mathcal{N}_p$: adding the point $\infty$ turns the normal to a~circle. The $m$ points $\{r_i\}$ are indexed cyclically mod $m$ by numbers $1,\dots,m$, which fits with the earlier defined notation: $r_i=r_{-m-1+i}$.

 This ordering also yields sections of the normal bundle which stratify the critical set $\mathcal{C}_e$. The regular points of $\mathcal{C}_e$ correspond to focal points with $\nu=1$, whose neighborhoods are smooth codimension-$1$ submanifolds. The singular points correspond to focal points of higher multiplicity.
The singularities of the focal set consist not only of the images of points with $\nu > 1$ but also of self intersections and singularities of the images of the points with $\nu = 1$.

In the case $\nu=1$ it follows that $d^2_y$ has a~singularity of corank 1 at $p$, so must be of type $A_k$ with $k \ge 2$. In the generic case this type is $A_2$. Also, the Jacobian of the evaluation map $e$ has corank 1. This means (see~\cite{GG}) that in the generic case $e$ is a~submersion with folds, and $\mathcal{F}_M$ is locally the image of the fold set.
It is also known that the normal $\mathcal{N}_p$ at a~$p$-focal points is tangent to the focal set (which is a~well known property for evolutes of curves); but if the type is $A_3$ or higher then the focal set is not necessarily smooth (as is known for the evolute).

The focal set $\mathcal{F}_M$ cuts the ambient space $\mathbb{R}^n$ into \textit{chambers} (that is, connected components of the complement of $\mathcal{F}_M$).
The type of the associated Morse functions $d^2_y$ depends only on the chamber containing $y$.

Transversal crossing of the focal hypersurface at its smooth $A_2$ point amounts to a~birth (or death) of two critical points of the function $d_y^2$ whose indices are $i$ and $i-1$.

We end this section with the folowing definition:

\begin{definition}
We say that the normal $\mathcal{N}_p$ is \textit{regular} if
\begin{itemize}
\item $\mathcal{N}_p$ does not intersect the singularities of $\mathcal{F}_M$,
\item $\mathcal{N}_p$ intersects the focal set transversally, except for the $p$-focal points,
\item $d_y^2$ is of type $A_2$ at the point $p \in M^m$ for all non-$p$-focal points
$y \in \mathcal{N}_p \cap \mathcal{F}_M$,
\item there are exactly $m$ different $p$-focal points on $\mathcal{N}_p$.
\end{itemize}
\end{definition}

 The last condition excludes $p$-focal points at infinity (if the Gauss curvature at $p$ is zero). For generic immersions almost every normal is regular. Regularity implies that the focal set is a~hypersurface close to the intersections with the normal.

\section{The concurrent normals theorem}

Let us recall that Morse functions $f:M^m \rightarrow \mathbb{R}$ on a~compact manifold satisfy the\textit{ Morse inequalities}:
The number of critical points of index $i$ does not exceed the Betti number $\beta_i(M^m,\mathbb{Z}_2)$. Therefore, the number of concurrent normals {from a~non-focal point $y$} is at least the sum of Betti numbers $\beta=\beta_0(M^m,\mathbb{Z}_2)+\dots+\beta_{m}(M^m,\mathbb{Z}_2)$.

\begin{definition}
A point $y \notin \mathcal{F}_M$ is called an \textit{excess-k point} if $y$ is an intersection point of at least $\beta +k$ normals from different points on $M^m$.
\end{definition}

\begin{definition}
A compact closed manifold $M^m$ is called \textit{$i$-trivial} (for some $0<i<m$), if the $i$-th Betti number vanishes: $\beta_{i} (M^m, \mathbb{Z}_2)=0$.
In this case, by Poincar\'{e} duality, we also have $\beta_{m-i} (M^m,\mathbb{Z}_2)=0$.
\end{definition}

\begin{theorem}\label{ThmMain}
Let $M^m \subset \mathbb{R}^n$ be an immersed $C^2$-smooth compact manifold. For any regular normal $\mathcal{N}_p$ holds:
  \begin{enumerate}
  \item\label{itm1}   $\mathcal{N}_p$
contains an intersection point of at least $\beta +2$ normals from different points on $M^m$.
 \item\label{itm2} If $M^m$ is an $i$-trivial manifold, then $\mathcal{N}_p$
contains an excess-4 point, i.e., the intersection of at least $\beta +4 $ normals.
\end{enumerate}
\end{theorem}

\begin{remark}
One can localize the positions of excess-$4$ points:
 there exist excess-$4$ points on the segment $[r_i(p),r_{i+1}(p)]$, and also on the segment $[r_{-i-1}(p),r_{-i}(p)]$, provided that $H_i(M^m, \mathbb{Z}_2)=0$ (here, we use our modulo $m$ notation for the indices). It is clear that unless $i=\frac{m}{2} $, there are two such segments, since $\beta_i=0$ implies $\beta_{m-i}=0$. If $i=\frac{m}{2} $ there is a single interval.
\end{remark}

\begin{remark}
J. Pardon's approach (\cite{P}) works for any $m$-sphere, convex or not:
{Given an immersion $j:S^m \rightarrow \mathbb{R}^n$, for any continuous $J:B^{m+1} \rightarrow \mathbb{R}^n$ such that $J$ restricted to $\partial B^{m+1}$ equals $j$, there exists an excess-$4$ point in $J(B^{m+1})$.}
\end{remark}

\begin{remark}
Tight embeddings and immersions are especially interesting: see the first reminder of section~\ref{rem}. In any case, Theorem~\ref{ThmMain} holds true for non-tight ones as well.
\end{remark}

\subsection*{Proof of Theorem~\ref{ThmMain}}\label{proof}

{For $y \notin \mathcal{F}_M$} we say \textit{the function $d_y^2$ has $q$ extra critical points of index} $i$ \ whenever the number of critical points of index $i$ is at least $\beta_i+q$.
\begin{lemma}\label{LemmaExcess}
In a~neighbourhood of the point $r_{\pm k}$ on the regular normal $\mathcal{N}_p$ there is an extra critical point of index $k$ and an extra critical point of index $k-1$.
\end{lemma}
\begin{proof}
The normal $\mathcal{N}_p$ is tangent to the focal hypersurface at the point $r_k$. While moving along the normal,
 the critical point $p$ persists, and its Morse index changes at the point $r_k$ from $k-1$ to $k$. Since a~deformation of the $A_2$ singularity can have only $0$ or $2$ Morse points, it follows that
$\mathcal{N}_p$ stays locally in the same chamber, except at the tangent point. At the tangent point $r_k$ it meets the closure of a~chamber with two fewer critical points, that is, of index $k$ and $k-1$. 
\end{proof}

Since each {regular} normal contains at least one $p$-focal point, the lemma implies the statement \ref{itm1} of the theorem.

 For a~regular normal $\mathcal{N}_p$, consider a~point $y=p+t\mathbf{n}$ moving along $\mathcal{N}_p$, starting at the point $p$.
We recall that $\mathcal{N}_p$ is compactified by adding a~point at $\infty$, so $y$ makes a~full turn and eventually returns to $p$.

The point $p$ is the first  minimum point of $d^2_y$, (its Morse index is zero, $\mu=0$), next after passing $r_1$ its Morse index turns to $1$, and so on. That is, after passing through $r_i$ we have $\mu=i$. Passing through infinity is (generically) not a~bifurcation point, but it changes each Morse index $\mu$ to $m - \mu$. In particular, minima and maxima exchange their roles. After that, the Morse index changes also in a~controlled way, that is, after passing through $r_{-i}$ we get { $\mu=1-i$.}

Combining the above with Lemma~\ref{LemmaExcess}, we get:
  \begin{enumerate}
        \item $d^2_y$ depends smoothly on $t$;
        \item $d^2_y$ is a~Morse function for each $y$ except for finitely many {bifurcation points};
        \item for a~regular normal each of the bifurcation points is $A_2$ for $d^2_y$ and has one of the two types:
          \begin{enumerate}
         \item A\textit{ birth/death point}: Transversal crossing of the focal hypersurface translates to two critical points with some neighbouring indices $k$ and $k-1$ colliding and disappearing, or conversely, two critical points with neighbouring indices appearing.
         \item An \textit{index exchange point}: A tangential point to the focal hypersurface translates to two critical points with neighbouring indices $i$ and $i-1$ colliding
         at $r_i$. There arises a~degenerate critical point of type $A_2$ which splits afterwards into two non-degenerate critical points with the same indices $i$ and $i-1$. Index exchange points are exactly the $p$-focal points $r_i$ and $r_{-i}$.
\end{enumerate}
\end{enumerate}

 Now we prove statement \ref{itm2} of the theorem.
Let
$H_i(M, \mathbb{Z}_2)$= 0. with $i\leq m/2$. Then either $r_i$ or $r_{-i}$ exists. Assume that $r_i$ exists, and further assume we have no excess-$4$ points on the segment $[r_i,r_{i+1}]$.
 The point $r_i$ is an index exchange point and by Lemma~\ref{LemmaExcess} the function $d_y^2$ has exactly one extra critical point of index $i$ and exactly one extra critical point of index
  $i-1$ around $r_i$. For the same reason there are two extra critical points of indices $i$ and $i+1$ around $r_{i+1}$.
  
Assume that $\beta+2$ is the maximal number of critical points on the segment
 $[r_i,r_{i+1}]$.
The (unique!) critical point of index $i$ is the point $p$ and it is not allowed to bifurcate. Indeed, its Morse index persists on $[r_i,r_{i+1}]$.
The extra critical point of index $i-1$ does not bifurcate since no partner for an $A_2$-bifurcation is available. The same holds around the point $r_{i+1}$ for the extra critical point of index $i+1$. 

Since a contradiction arises from having two extra critical points of indices $i-1$ and $i+1$, there exists at least one excess-$4$ point on the interval $[r_i,r_{i+1}]$.
We note that the segment $[r_i,r_{i+1}]$ might contain infinity, but this fact poses no complications, since
passing through $\infty$ maintains the number of critical points but inverts their indices.


\section{Remarks}\label{rem}

\subsection*{A reminder on tight immersions}

Given an immersed manifold $M^m$, the limit to infinity of the quadratic distance function $d_y^2$ along a~normal with direction $\mathbf n$ is the linear function $l_{\mathbf n}: M^m \to \mathbb{R}$ defined by $l_{\mathbf n}(x) = \langle x, \mathbf n \rangle$. The functions $d_y^2$ and $l_{\mathbf n}$ have the same number of critical points with the same Morse indices when $y$ is close enough to infinity. An immersion $M^m $ in $ \mathbb{R}^n$ is called \textit{tight} if the number of critical points of $l_{\mathbf n}$ equals $ \beta (M^m,\mathbb{Z}_2)$
for almost all $\mathbf{n}$. Tightness is a~generalization of convexity; a~tight embedding of the $(n-1)$-sphere as a~hypersurface is necessarily convex. Tight immersions, including the relation to total absolute curvature have been studied intensively by Chern-Lashoff, Kuiper, Banchoff and Kuehnel. See e.g.~\cite{CR1, CR2, K, Kui}.

Some useful examples:
 There exist tight immersions in $\mathbb{R}^3$ for all oriented surfaces and non-oriented surfaces with Euler characteristic smaller than $- 1$.
The projective plane, the Klein bottle and the projective plane with one handle cannot be tightly immersed in $\mathbb{R}^3$.
Examples of $4$-dimensional $i$-trivial manifolds with (tight) embedding in $\mathbb{R}^5$ are
 $S^4$, $S^2 \times S^2$ and their connected sums~\cite{K}.

\subsection*{On non-regular normals}
For generic immersions there is a~dense set of regular normals. Non-regular normals are, for instance, those based at umbilical points and based at points with zero Gaussian curvature, where one of the centers of curvature is at infinity. But there are many cases where regular normals do not occur. We mention here the round sphere and cylinder. This happens also in the case where $d^2_y$ has exactly $\beta$ critical points for a~dense set of $y$. Such immersions are called \textit{taut immersions}, which are very special, see e.g.~\cite{CR1}. Taut is much more restrictive than tight. Their focal sets are very degenerate.

If parts of the focal set have dimension $\le n-2$, then we will get an infinite number of normals from these focal points. To be more precise: Consider the evaluation map $e$ from one of the $n-1$ dimensional strata of $\mathcal{C}_e$. If this map reduces dimension, then for an image point $y$ the quadratic distance $d^2_y$ has non-isolated singularities.

Note also that for certain non-regular normals the proof of theorem~\ref{ThmMain} still applies, e.g. when $M$ is $i$-trivial and the focal points $r_i$ and $r_{i+1}$ have multiplicity $\nu=1$.

\subsection*{Tubes over an immersed manifold} We describe below a~higher excess estimate for tubes.
Given an immersion $M^m \subset \mathbb{R}^n$, the \textit{tube of $M^m$} with radius $r> 0$ is defined as the restriction of the evaluation map $e$ to the sphere bundle $S_r(M)$, consisting of normal vectors of length $r$. For sufficiently small $r$ we get an immersion
 $S_r(M) \subset \mathbb{R}^n $.

 \begin{proposition}
 Let $y \notin \mathcal{F}_M \cup M^m$ be an excess-$k$ point for $M^m$. Then $y$ is an excess-$2k$ point for $S_r(M)$.
 \end{proposition}

 \begin{proof}
 Any normal $\mathcal{N}$ from a~point $y \in \mathbb{R}^n$ to $M^m$ is also a~normal to $S_r(M)$. A single intersection point $p \in M^m$ will give precisely two intersection points $p_1,p_2$ with $S_r(M)$. The normal $\mathcal{N}$ is a~normal for the points $p_1,p_2$, so the total number of critical points of $d^2_y$ doubles. The observation $\beta(S_r(M),\mathbb{Z}_2)= 2 \beta (M,\mathbb{Z}_2) $ from~\cite{CR1} completes the proof.
 \end{proof}
 
As a~corollary: taut immersions give taut tubes (see \cite{CR1}). Note that points of $M$ produce non-isolated singularities for $d^2_y$.

 For tubes, we can give more information about the locations and indices of the excess points than in the proof of
Theorem~\ref{ThmMain}.
 The focal set of a~tube $S_r(M)$ consists of two pieces:
$ \mathcal{F}_M \cup M^m$ (\cite{CR1}, Theorem~3.3). The focus points on $M^m$ all have multiplicity $\nu =n-m $. For $n-m \ge 2$ no normal to the tube can be regular.
We can still compute the index of $d_y^2$ at $p_i$ by counting the $p$-focal points with multiplicity on the coinciding normals $\mathcal{N}_{p_1}$ and $\mathcal{N}_{p_2}$. By traveling over the two normals as in the proof of Theorem~\ref{ThmMain} we see that the excess doubles. This happens in the same intervals as for $M^m$.

\subsection*{Acknowledgements}
The authors acknowledge useful discussions with George Khimshiashvili and hospitality
of CIRM, Luminy, where this research was completed in framework of "Research in Pairs"
program in November, 2022.

\EditInfo{January 21, 2023}{April 16, 2023}{Jacob Mostovoy and Sergei Chmutov}


{\small
\begin{thebibliography}{10}

\bibitem{GP}
G.~P. A.~Grebennikov.
\newblock A note on the concurrent normal conjecture.
\newblock {\em Acta Mathematica Hungarica}, 167:529--532, 2022.

\bibitem{H2}
E.~Heil.
\newblock Existenz eines 6-normalenpunktes in einem konvexen k\"{o}rper.
\newblock {\em Arch. Math. (Basel)}, 32(4):412--416, 1979.

\bibitem{H1}
E.~Heil.
\newblock Concurrent normals and critical points under weak smoothness
  assumptions.
\newblock {\em Annals of the New York Academy of Sciences}, 440(1):170--178,
  1985.

\bibitem{K}
W.~K\"{u}hnel.
\newblock Tight embeddings of simply connected 4-manifolds.
\newblock {\em Documenta Mathematica}, 9:401--412, 2004.

\bibitem{Kui}
N.~H. Kuiper.
\newblock Minimal total absolute curvature for immersions.
\newblock {\em Invent. Math.}, 10:209--238, 1979.

\bibitem{GG}
V.~G. M.~Golubitski.
\newblock Stable mappings and their singularities.
\newblock In {\em Graduate Texts in Mathematics}. Springer, 1973.

\bibitem{M-M}
Y.~Martinez-Maure.
\newblock On the concurrent normals conjecture for convex bodies, 2021.

\bibitem{M}
J.~Milnor.
\newblock {\em Morse theory}.
\newblock Princeton University Press, 1963.

\bibitem{P}
J.~Pardon.
\newblock Concurrent normals to convex bodies and spaces of morse functions.
\newblock {\em Math. Ann.}, 352(1):55--71, 2012.

\bibitem{Po}
I.~R. Porteous.
\newblock {\em Geometric Differentiation for the intellegence of curves and
  surfaces}.
\newblock Cambridge University Press, 1994.

\bibitem{CR1}
R.~T.E.Cecil.
\newblock {\em Tight and Taut Immersions of Manifolds}.
\newblock Pitman, 1985.

\bibitem{CR2}
R.~T.E.Cecil.
\newblock {\em Geometry of Hypersurfaces}.
\newblock Springer, 2015.

\bibitem{A}
S.-Z. V.Arnol'd, A.Varchenko.
\newblock {\em Singularities of differentiable maps I and II}.
\newblock Birkh\"{a}user, 1985.

\end{thebibliography}
}
\end{document}